\documentclass[letterpaper,conference,10pt]{ieeeconf}
\usepackage{booktabs}
\usepackage{amsmath,amssymb}
\usepackage{subfigure}
\usepackage{stfloats}
\usepackage{todonotes}
\usepackage[hidelinks]{hyperref}
\usepackage{graphicx}
\usepackage{epstopdf}
\usepackage{tikz}
\usepackage{lipsum}
\usepackage{circuitikz}
\usetikzlibrary{decorations.text}
\usepackage{pgfplots}
\usepgfplotslibrary{polar}
\pgfplotsset{compat=newest}
\usepackage{cite}
\usepackage{stfloats}
\usepackage{MnSymbol}
\usepackage{lipsum}
\usepackage{amsthm,dsfont}
\newtheorem{defn}{Definition}

\newtheorem{lem}[defn]{Lemma}

\newtheorem{assum}[defn]{Assumption}
\newtheorem{ex}[defn]{Example}
\newtheorem{thm}[defn]{Theorem}

\providecommand{\R}{\ensuremath \mathbb{R}}
\providecommand{\N}{\ensuremath \mathbb{N}}
\providecommand{\ip}[1]{\ensuremath \langle #1\rangle}
\providecommand{\spt}{\ensuremath \text{spt}}

\tikzstyle{stuff_nofill}=[rectangle,draw,font={A}]
\tikzstyle{stuff_fill}=[rectangle,draw=none,fill=white,font={A}]

\title{Convex Estimation of the $\alpha$-Confidence Reachable Sets of Systems with Parametric Uncertainty}

\author{Patrick~Holmes,~Shreyas~Kousik,~Shankar~Mohan~and~Ram~Vasudevan
 \thanks{S.~Mohan is with the Department of Electrical Engineering and Computer Science, University of Michigan, Ann Arbor, MI 48109
{\scriptsize \texttt{elemsn@umich.edu}}}
 \thanks{P.~Holmes, S.~Kousik \& R.~Vasudevan are with the Department of Mechanical Engineering, University of Michigan, Ann Arbor, MI 48109
{\scriptsize \{\texttt{pdholmes,~skousik,~ramv}\}~\texttt{@umich.edu}}}
}

\IEEEoverridecommandlockouts                              

\usepackage[T1]{fontenc}

\begin{document}

\maketitle
\begin{abstract}
	Accurately modeling and verifying the correct operation of systems interacting in dynamic environments is challenging.
	By leveraging parametric uncertainty within the model description, one can relax the requirement to describe exactly the interactions with the environment; however, one must still guarantee that the model, despite uncertainty, behaves acceptably. 
	This paper presents a convex optimization method to efficiently compute the set of configurations of a polynomial dynamical system that are able to safely reach a user defined target set despite parametric uncertainty in the model.
	Since planning in the presence of uncertainty can lead to undesirable conservativeness, this paper computes those trajectories of the uncertain nonlinear systems which are $\alpha$-probable of reaching the desired configuration.
	The presented approach uses the notion of occupation measures to describe the evolution of trajectories of a nonlinear system with parametric uncertainty as a linear equation over measures whose supports coincide with the trajectories under investigation.
	This linear equation is approximated with vanishing conservatism using a hierarchy of semidefinite programs each of which is proven to compute an approximation to the set of initial conditions that are $\alpha$-probable of reaching the user defined target set safely in spite of uncertainty.
	The efficacy of this method is illustrated on four systems with parametric uncertainty.
\end{abstract}
\section{Introduction}

Verifying the correct operation of systems interacting in dynamic environments is challenging.
In fact, the difficulties associated with modeling such systems exactly compounds this verification challenge.
By introducing parametric uncertainty within the model, one can compensate directly for the inability to construct exact models; however, to ensure the satisfactory operation of uncertain systems one must provide systematic guarantees on all probable behaviors.
Unfortunately, unforeseen conservativeness may arise when certain low probability outcomes restrict the potential behavior of the system.
To address this shortcoming, this paper presents an approach to compute the set of initial conditions of a nonlinear system with parametric uncertainty that are at least $\alpha$-probable of arriving at a user defined target set. 

A variety of numerical methods have been proposed to verify the satisfactory operation of nonlinear systems with parametric uncertainty. 
The most popular of these approaches have relied upon generating or evaluating pre-constructed Lyapunov functions to compute the domain of attraction of an uncertain system \cite{chesi2004estimating,topcu2007stability}.
This has required checking Lyapunov's criteria for polynomial systems by using sums-of-squares programming, which results in a bilinear optimization problem that is usually solved using some form of alternation \cite{prajna2007convex}. 
However, such methods are not guaranteed to converge to global optima (or necessarily even local optima), and require feasible initializations.

Others have developed tools to perform safety verification of more general stochastic nonlinear dynamical systems \cite{mitchell2005time,prajna2007framework}. 
Hamilton-Jacobi Bellman based approaches, for example, have also been applied to compute the uncertain backwards reachable set for nonlinear systems with arbitrary uncertainty affecting the state at any instance in time \cite{mitchell2005time}.
These approaches solve a more general problem and scale well despite state space discretization when the specific system under consideration has special structure \cite{maidens2013lagrangian}.
Barrier certificate methods \cite{prajna2007convex} have also been utilized to perform stochastic safety verification by using a super martingale. 

This paper leverages a method developed in a recent paper that describes the evolution of trajectories of an uncertain dynamical system using a linear equation over measures \cite{mohan2016convex}. 
As a result of this characterization, the set of configurations that are able to reach a target set despite parametric uncertainty, called the \emph{uncertain backwards reachable set}, can be computed as the solution to an infinite dimensional linear program over the space of nonnegative measures.
This approach, which was inspired by several recent papers \cite{henrion2014convex,majumdar2014convex,shia2014convex}, computes an approximate solution to this infinite dimensional linear program using a sequence of finite dimensional relaxed semi-definite programs via Lasserre's hierarchy of relaxations \cite{lasserre2001global} that each satisfy an important property:
each solution to this sequence of semi-definite programs is an outer approximation to the uncertain backwards reachable set with asymptotically vanishing conservatism.
Our approach will utilize this same formulation to construct an outer approximation to the set of $\alpha$-probable points in the uncertain backwards reachable set which we call the \emph{$\alpha$-level backwards reachable set}.

This approach of characterizing the behavior of the system using an infinite dimensional program over measures has also been used to perform safety verification of stochastic nonlinear systems \cite{sloth2015safety}. 
In that instance, initial conditions of the stochastic system whose trajectories on \emph{average}
have probability higher than some user-specified $p$ of arriving at some target set were computed using a semidefinite programming hierarchy.
In this paper, we consider instead the problem of determining which set of initial conditions of a dynamical system have a user-specified probability of arriving at a target set under parametric uncertainty within the model. 
Since there is no stochastic behavior in the dynamical system, our approach does not consider an average probability over each trajectory.

The remainder of the paper is organized as follows:
Section \ref{sec:preliminaries} introduces the notation used in the remainder of the paper, the class of systems under consideration, and the backwards reachable set problem under parametric uncertainty;
Section \ref{sec:prob} describes how the $\alpha$-level backwards reachable set under parametric uncertainty is the solution to an infinite dimensional linear program;
Section \ref{sec:implementation} constructs a sequence of finite dimensional semidefinite programs that outer approximate the infinite dimensional linear program with vanishing conservatism;
Section \ref{sec:examples} describes the performance of the approach with three examples;
and, Section \ref{sec:conclusion} concludes the paper.

\section{Preliminaries}
\label{sec:preliminaries}

This section describes the class of systems under consideration and outlines the problem of interest.
\subsection{Notation}
In the remainder of this text the following notation is adopted: sets are italicized and capitalized (ex. $K$). The set of continuous functions on a compact set $K$ are denoted by $\mathcal C(K)$.
The ring of polynomials in $x$ is denoted by $\R[x]$, and the degree of a polynomial is equal to the degree of its largest multinomial; the degree of the multinomial $x^\alpha,\,\alpha\in \N_{\ge 0}^n$ is $|\alpha|=\|\alpha\|_1$; and $\R_d[x]$ is the set of polynomials in $x$ with maximum degree $d$.
The dual to $\mathcal C(K)$ is the set of Radon measures on $K$, denoted as $\mathcal M(K)$, and the pairing of $\mu\in \mathcal M(K)$ and $v\in \mathcal C(K)$ is:
  \begin{align}
  \ip{\mu,v}=\int_{K}v(x)\,d\mu(x).
  \end{align}
We denote the nonnegative Radon measures by ${\cal M}_+(K)$. The space of Radon probability measures on $K$ is denoted by ${\cal P}(K)$.
The Lebesgue measure is denoted by $\lambda$. Finally, the support of measures, $\mu$, is identified as $spt(\mu)$.
\subsection{System class}
In this paper, we restrict our attention to the class of parametrically uncertain drift systems; i.e. systems of the following form:
\begin{align}
  \dot {x} = f(x,\theta),
\end{align}
where $x\in X$, are the states of the system, and $\theta\in \Theta$ are uncertain parameters.
\begin{assum}
  $X$ and $\Theta$ are compact, and $f$ is Lipschitz continuous in $x$ and $\theta$.
\end{assum}
\begin{ex}[Van der Pol Oscillator]
Consider the uncertain Van der Pol oscillator whose dynamics is:
  \begin{align}
  \begin{aligned}
  \dot x_1 =&\,    -2x_2\\
  \dot x_2=&\,    0.8x_1+(9+5\theta)x_2(x_1^2-0.21)
    \end{aligned}
\end{align}
where $\theta \in [-0.5,0.5]$ and $x(t)\in X=[-1,1]$. In this example, the limit cycle of the Van der Pol oscillator deforms as the value of $\theta$ changes: the larger the value of $\theta$, the smaller the volume of the area contained inside the limit cycle.
\end{ex}
Parameters $\theta\in \Theta$ are assumed to be drawn according to a probability distribution $\mu_\theta\in \mathcal P(\Theta)$.
\begin{assum}
  If the uncertain parameter $\theta$ is distributed according to $\mu_\theta$, $\mu_\theta$ is absolutely continuous with respect to the Lebesgue measure. We denote this by: $\mu_\theta \ll \lambda_\theta$, where $\lambda_\theta$ is the Lebesgue measure on $\spt(\mu_\theta)$.
\end{assum}
It is assumed that the unknown parameters do not change with time and that they are instantiated at time $t=0$. That is, the uncertain system can be thought to evolve in the embedded space $X\times \Theta$ according to the following dynamics
\begin{align}
    \begin{bmatrix}
    \dot {x}\\\dot \theta
  \end{bmatrix} = \begin{bmatrix}
    f(x,\theta)\\0
  \end{bmatrix}.
  \label{eq:aug_system}
\end{align}
For notational convenience, we denote the unique solution to the dynamics in Eqn.~(\ref{eq:aug_system}) as the absolutely continuous function $\gamma$ defined as follows:
\begin{align}
  \gamma : [0,T]\xrightarrow{\text{Eqn.}~(\ref{eq:aug_system})} X\times \Theta,\phantom{3} \gamma(0) = [x;\theta].
\end{align}
In addition, let us denote by $\mathcal T$, the time interval $[0,T]$.
\subsection{Problem Description}
The objective of this paper is to identify the $\alpha$-confidence, time limited reachable set, or \emph{uncertain backwards reachable set} (BRS) of $X_T$. This definition relies on the following set-valued mapping from $X$ to the Borel $\sigma$-algebra on $\Theta$ (closed sets), $\mathcal B(\Theta)$:
\begin{align}
\begin{split}
  \Gamma(x)=\{\theta \in \Theta\mid \exists\, \gamma :\mathcal T\xrightarrow{\text{Eqn.}~(\ref{eq:aug_system})} X\times \Theta, \text{ with }\\
  \gamma(0)= [ x ;\theta],\gamma(T)\in X_T\times \Theta\}.
\end{split}
\label{eq:Gamma}
\end{align}
For a given value of $x\in X$, $\Gamma(x)$ is the set of distinct values of the parameter $\theta$, such that the solution trajectories of the system in Eqn.~(\ref{eq:aug_system}), with the states initialized to $[x;\Gamma(x)]$ arrives at $X_T$ at time $T$.
\begin{defn}
\label{defn:brs}
  The $T$-time $\alpha$-confidence backwards reachable set of $X_T$, under the dynamics in Eqn.~(\ref{eq:aug_system}), is the defined as follows
  \begin{align}
    X_0^{\alpha} = \{x_0\mid \mu_\theta(\Gamma(x_0))\ge \alpha\}
  \end{align}
\end{defn}
The $\alpha$-confidence BRS is the set of initial values of $x$ such that for each $x$, the {\em mass} of $\Gamma(x)$ under $\mu_\theta$ is larger than $\alpha$; i.e. the set of initial conditions for which the probability of arriving in $X_T$ at $t=T$ is greater or equal to $\alpha$.
The remainder of this paper is devoted to tractably computing $X_0^{\alpha}$.

\section{Problem formulation}
\label{sec:prob}
In this section, we present a methodology to compute the time limited $\alpha$-confidence backwards reachable set of dynamic systems. The proposed methodology consists of two steps: (1) estimating the set of all feasible initial conditions of the system in Eqn.~(\ref{eq:aug_system}) such that $\gamma(T)\in X_T\times \Theta$; (2) determining the subset of initial conditions that reach the target set with desired probability. Step (1) is addressed by solving an infinite dimensional problem, and step (2) requires integrating the optimal solution of step (1).

To estimate the BRS, we use the notion of {\em occupation measures} \cite{pitman1977}. Given an initial condition for the system, the occupation measure evaluates to the amount of time spent by the resultant trajectory in any subset of the space. The occupation measure $\mu(\cdot\mid x_0,\theta)\in \mathcal M_+(\mathcal T\times X\times \Theta\mid x_0,\theta)$ is formally defined as follows:
\begin{align}
  \mu(A\times B\times C\mid x_0,\theta) = \int_{0}^T I_{A\times B\times C}(t,x,\theta\mid x_0,\theta)\,dt,
  \label{eq:occ}
\end{align}
where $I_{A}(x)$ is the indicator function on the set $A$ that returns one if $x\in A$ and zero otherwise. With the above definition of the occupation measure, using elementary functions, it can be shown that:
\begin{align}
  \ip{\mu(\cdot\mid x_0,\theta),v} = \ip{\lambda_t,v(t,x(t\mid x_0,\theta),\theta)},
  \label{eq:ip}
\end{align}
where $\lambda_t$ is the Lebesgue measure on $\mathcal T$.

The {\em occupation measure} has an interesting characteristic -- it completely characterizes the solution trajectory of the system resulting from an initial condition. Observe that the occupation measure as defined in Eqn.~(\ref{eq:occ}) is conditioned on the initial values of states and parameters. Since we are interested in the collective behavior of a set of initial conditions, we define the {\em average} occupation measure as:
\begin{align}
\mu(A\times B\times C) = \int_{X\times \Theta} \mu(A\times B\times C\mid x,\theta)\,d\mu_0,
\end{align}
where $\mu_0$ is the un-normalized distribution of initial conditions. The value to which the {\em average} occupation measure evaluates over a given set in $\mathcal T\times X\times \Theta$ can be interpreted as the cumulative time spent by all solution trajectories which begin in $\spt(\mu_0)$.
\par
Given a test function $v\in \mathcal C^1(\mathcal T\times X\times \Theta)$, using the Fundamental Theorem of Calculus, its value at time $t=T$ is given by
\begin{align}
\begin{split}
  v(T,x(T\mid x_0,\theta),\theta) =&\, v(0,x_0,\theta)\\
  +&\,\int\limits_{0}^T\hspace*{-.08in}\left(\frac{\partial v}{\partial x}\cdot f+\frac{\partial v}{\partial t}\right)(t,x(t\mid x_0,\theta),\theta)\,dt
\end{split}
  \label{eq:ftc}
\end{align}
Using the relation defined in Eqn.~(\ref{eq:ip}) and by defining the linear operator $\mathcal L_f$ on $\mathcal C^1$ functions (Lie derivative) as the following:
\begin{align}
  \mathcal L_{f}v = \frac{\partial v}{\partial x}\cdot f+ \frac{\partial v}{\partial t},
\end{align}
Eqn.~(\ref{eq:ftc}) is re-written as
\begin{align}
\begin{split}
  v(T,x(T\mid x_0,\theta),\theta) =&\, v(0,x_0,\theta)\\
  +&\,\int\limits_{X\times \Theta}\mathcal L_{f}v\,d\mu(t,x,\theta\mid x_0,\theta).
\end{split}
\label{eq:ftc2}
\end{align}
Integrating Eqn.~(\ref{eq:ftc2}) with respect to $\mu_0$, the distribution of initial conditions, and defining a new measure $\mu_T\in \mathcal M_+ (X_T\times \Theta)$, as the following
\begin{align}
  \mu_T(A\times B) = \int\limits_{X\times \Theta}I_{A\times B}(x(T\mid x_0,\theta),\theta)\,d\mu_0,
\end{align}
produces the following equality
\begin{align}
  \ip{\delta_T\otimes \mu_T,v}= \ip{\delta_0\otimes \mu_0,v}+\ip{\mu,\mathcal L_{\tilde f}v},
  \label{eq:lv}
\end{align}
where, with a slight abuse of notations, $\delta_t$ is used to denote a Dirac measure situated at time $t$. Using adjoint notations, Eqn.~(\ref{eq:lv}) can be written as:
\begin{align}
  \delta_T\otimes \mu_T= \delta_0\otimes \mu_0+\mathcal L_{f}'\mu
  \label{eq:adjointlv}
\end{align}
Equation~(\ref{eq:adjointlv}) is a version of the Liouville equation, holds for all test function $v\in \mathcal C^1(\mathcal T\times X\times \Theta)$, and summarises the visitation information of all trajectories that emanate from $\spt(\mu_0)$ and terminate in $\spt(\mu_T)$. Several recent papers provide a more detailed discussion on the Liouville equation \cite{henrion2014convex},\cite{mohan2016convex}.
\par
Within this framework of measures and the Liouville Equation, we first formulate the problem of identifying the set of all pairs $(x_0,\theta)$ such that the solution trajectory of the system in Eqn.~(\ref{eq:aug_system}) initialized at $[x;\,\theta]$ arrives at $X_T\times \Theta$ at $t=T$. This problem can be interpreted as one that attempts to identify the largest support for $\mu_0$ that ensures the existence of measures $\mu$ and $\mu_T$ such that $(\mu_0,\mu_T,\mu)$ satisfy Eqn.~(\ref{eq:lv}). To measure the size of $\spt(\mu_0)$, we use the Lebesgue measure on $X\times \Theta$, $\lambda_x\otimes \lambda_\theta$.
\par
This problem is posed as an infinite dimensional Linear Program (LP) on measures as defined below.
%
%
\begin{flalign}\nonumber
    & & \sup_{\Lambda} \hspace*{1cm} & \ip{\mu_{0},\mathds 1} && (P)\\
    & & \text{st.} \hspace*{1cm} &\mu_{0}+\mathcal L_{f}'\mu=\,\mu_{T} && \label{eq:primal:liouville}\\
    & & & \mu_{0}+\hat\mu_{0}=\,\lambda_x\otimes \lambda_\theta &&
\end{flalign}
where $\lambda_x\otimes \lambda_\theta$ is the Lebesgue measure supported on $X\times \Theta$, $\Lambda:=(\mu_0,\hat\mu_{0},\mu_T) \in {\mathcal M}_+({\mathcal T} \times X \times \Theta) \times \mathcal M_+( X \times \Theta)\times \mathcal M_+( X_{T}\times \Theta )$ and $\mathds 1$ denotes the function that takes value $1$ everywhere.
\begin{lem}
  The support of $\mu_0$, $\spt(\mu_0)\subset X\times \Theta$ is the largest collection of pairs $(x,\theta)$ that, if used as initial conditions to Eqn. \eqref{eq:aug_system}, produce a solution trajectory that terminates in $X_T\times \Theta$ at $t=T$.
\end{lem}
\par
The dual problem, on continuous functions, corresponding to $(P)$ is the following
\begin{flalign}\nonumber
  & &  \inf_{\Xi}&\phantom{3} \ip{\lambda_x\otimes \lambda_\theta,w}          && (D) \nonumber \\
  & &  \text{st.}&\phantom{3}\mathcal L_{f} v(t,x,\theta)\le 0         && \forall (t,x,\theta)\in \mathcal T\times X\times\Theta\\
  & &  &\phantom{3} w(x,\theta) \ge 0                                         && \forall (x,\theta)\in X\times\Theta\\
  & &  &\phantom{3} w(x,\theta) - v(0,x,\theta) - 1 \ge 0                     && \forall (x,\theta)\in X\times\Theta\\
  & & & \phantom{3} v(T,x,\theta)\ge 0                                        && \forall (x,\theta)\in X_T\times \Theta
\end{flalign}
where $\Xi := (v,w)\in \mathcal C^1(\mathcal T\times X\times \Theta)\times \mathcal C(X\times \Theta)$. The solution to $(D)$ has an interesting interpretation: $v$ is similar to a Lyapunov function for the system, and $w$ resembles an indicator function on $\spt(\mu_0)$. Moreover:
\begin{lem}
  There is no duality gap between problems $(P)$ and $(D)$.
\end{lem}
\begin{lem}
  Given the pair of functions $(v^*,w^*)$ which is the optimal solution to $(D)$, the 1-super-level set of $w^*$ contains $\spt(\mu_0)$.
\end{lem}
The following salient result on the shape of $w$ is the critical result that we will employ to estimate $X_0$ as defined in Defn.~\ref{defn:brs}.
\begin{thm}\cite[Theorem~3]{henrion2014convex}
\label{thm:indicator}
 There is a sequence of feasible points to $(D)$ whose $w$ component converges uniformly in the $L^1$ norm to the indicator function on $\spt(\mu_0)$.
\end{thm}
An immediate consequence of the above theorem is that we can assume that $w$ evaluates to one on $\spt(\mu_0)$ and zero elsewhere. Now, relating the definition of $\Gamma$ in Eqn.~(\ref{eq:Gamma}) and this feature of $w^*$, the $w$-component of the optimal solution of $(D)$, one arrives at the following result that provides a means to compute the $\alpha$-confidence time limited reachable set, $X_0^\alpha$.
\begin{lem}
Suppose $(v^*,w^*)$ is an optimal solution of $(D)$. The $\alpha$-level backwards reachable of the set $X_T$ under the system dynamics of Eqn. (4) and parametric uncertainty with distribution $\mu_\theta$ is given by the following set
\begin{align}
X_0^\alpha:=\bigg\{x\,\big \mid \int_\Theta w^*(x,\theta)\,d\mu_\theta\ge \alpha \bigg\}
\end{align}
\label{lem:confidence}
\end{lem}
\begin{proof}
Using the provided information, define the measure $\eta\in \mathcal M_+(X\times \Theta)$ as follows:
\begin{align}
  \eta(A\times B) = \int_{A\times B} w^*(x,\theta) \,d(\lambda_x\otimes \mu_\theta).
\end{align}
It should be noted that $\Gamma(x)$ as defined in Eqn.~(\ref{eq:Gamma}) is the support of the conditional distribution of $\theta$ given $x$ of $\eta$; $\spt(\eta(\cdot\mid x))$. This is true since $w^*$ is an indicator function on $\spt(\mu_0)$, the set of all feasible pairs $(x,\theta)$.
In addition, by definition, $\lambda_x\otimes \mu_\theta \gg \eta$ and hence $\lambda_x\gg \pi^x_*\eta$, where $\pi^x_*\eta$ is the push-forward measure of $\eta$ under the $x$-projection operation as per \cite{lee2003smooth}. Thus, $\phi(x)$, $\lambda_x$ measurable, is the Radon-Nikodym derivative between $\lambda_x$ and $\pi_*^{x}\eta$ such that
\begin{align}
  \pi^x_*\eta(A) =&\, \int_A \phi(x)\,d\lambda_x,\phantom{4}\forall A\subset X.
\end{align}
The function $\phi(x)$, for each $x$ evaluates to the probability that $x$ will reach $X_T$ given the distribution of uncertainty. To see this, observe that
\begin{align}
  \int_A \phi(x)\,d\lambda_x = &\,\int_A \int_\Theta w^*\,d\,\mu_\theta d\lambda_x,\\
  =&\, \int_A \int_{\Gamma(x)} w^*\,d\,\mu_\theta d\lambda_x,\\
  =&\, \int_A \int_{\Gamma(x)} 1\,d\,\mu_\theta d\lambda_x,
\end{align}
where we have used the fact that $w^*\ge 1,\, \forall(x,\theta)\in \spt(\mu_0)$ (from Thm.~\ref{thm:indicator}), that $\mu_\theta$ is a probability measure, and Defn.~\ref{defn:brs}. The statement of the Lemma now follows from Defn.~\ref{defn:brs}.
\end{proof}
Lemma~\ref{lem:confidence} provides a means to compute the $\alpha$-confidence $T$-time backwards reachable set; one just has to integrate the $w$ component of the optimal solution to $(D)$ with respect to the distribution of $\theta$ and identify level sets. Solving the infinite dimensional problem is nontrivial; in the following section, we employ Lasserre's hierarchy of relaxations to arrive at a sequence of outer approximations of $X_0^\alpha$.

\section{Numerical Implementation}
\label{sec:implementation}

In this section, a sequence of Semidefinite Programs (SDPs) that approximate the solution to the infinite dimensional primal and dual defined in Sec.~\ref{sec:prob} are introduced.

\subsection{Lasserre's relaxations}
This sequence of relaxations is constructed by characterizing each measure using a sequence of moments\footnote{The $n$th moment of a measure ($\mu$) is obtained by evaluating the following expression
  $$y_{\mu,n}=\ip{\mu,x^n}.$$}
and assuming the following:
\begin{assum}
The dynamical system in Eqn. \eqref{eq:aug_system} is a polynomial.
Moreover the domain, the set of possible values of uncertainties, and the target set are semi-algebraic sets.
  \label{assump:poly}
\end{assum}
Recall that polynomials are dense in the set of continuous functions by the Stone-Weierstrass Theorem so this assumption is made without too much loss of generality.
\par
Under this assumption, given any finite $d$-degree truncation of the moment sequence of all measures in the primal $(P)$, a primal relaxation, $(P_d)$, can be formulated over the moments of measures to construct an SDP.
The dual to $(P_d)$, $(D_d)$, can be expressed as a sums-of-squares (SOS) program by considering $d$-degree polynomials in place of the continuous variables in $D$.
\par
To formalize this dual program, first note that a polynomial $p \in \R[x]$ is SOS or $p \in \text{SOS}$ if it can be written as $p(x) = \sum_{i=1}^m q_i^2(x)$ for a set of polynomials $\{q_i\}_{i=1}^m \subset \R[x]$.
Note that efficient tools exist to check whether a finite dimensional polynomial is SOS using SDPs~\cite{parrilo2000structured}.
Next, suppose we are given a semi-algebraic set $A = \{x \in \R^n \mid h_{i}(x) \geq 0, h_i \in \R[x], \forall i \in \N_m \}$.
We define the $d$-degree {\em quadratic module} of $A$ as:
\begin{align}
  \begin{split}
  Q_d(A)=\bigg\{q\in \R_d[x]\,\bigg|\, \exists \{s_k\}_{k \in \{0,1,...,m\} \cup \{0\}} \subset \text{SOS s.t. } \\
  q=s_0+\sum_{k\in \{1,...,m\}}h_{k}s_k \bigg\}
  \end{split}
\end{align}

\noindent The $d$-degree relaxation of the dual, $D_d$, can now be written as:
  \begin{flalign}\nonumber
    & & \inf_{\Xi_d} \hspace*{0.1cm} & \int_{X\times \Theta}w_d(x,\theta)\,d(\lambda_x\otimes\lambda_\theta) && \hspace*{-0.3cm} (D_d) \nonumber\\
    & & \text{st.} \hspace*{0.1cm} & w_d\in Q_d(X\times \Theta) && \\
    & & & v_d(T,x,\theta)\in Q_d(X_T \times \Theta) && \\
    & & & -\mathcal L_{f}v_d(t,x,\theta)\in Q_d(\mathcal T\times X \times \Theta) && \\
    & & & w_d-v_d(0,x,\theta)-1\in Q_d(X\times \Theta)
  \end{flalign}
where $\Xi_d=\Big\{ \big(v_d,w_d\big) \in \R_d[t,x,\theta]\times \R_d[x,\theta]\Big\}$.
A primal can similarly be constructed, but the solution to the dual can be used to directly generate a sequence of outer approximations to the uncertain backwards reachable set:
\begin{lem}
Let $w_d$ denote the $w$-component of the solution to $(D_d)$. Then $X_{(0,d)} = \{(x,\theta) \in X\times \Theta \mid w_d(x,\theta) \geq 1 \}$ is an outer approximation to $\spt(\mu_0)$ and $\lim_{d\to\infty}\lambda_x\otimes \lambda_\theta({ X}_{(0,d)} \backslash \spt(\mu_0)) = 0$.
\end{lem}
Using the finite-degree truncation of the infinite dimensional problem presented above, one can solve for a sequence of convergent approximations of the support of $\mu_0$. To approximate the BRS as defined in Defn.~\ref{defn:brs}, we need to perform an additional step, described in the next section.

\subsection{Generating outer approximations of $X_0^\alpha$}
This section presents two methods to use the outer approximations of $\spt(\mu_0)$ derived by solving $(D_d)$, to estimate $X_0^\alpha$ as defined in Defn.~\ref{defn:brs}. The first method relies on discretizing the state space and computing the probability that each node in the mesh will reach the target set $X_T$ at $t=T$, through Monte Carlo simulation. The second method poses an additional optimization problem over polynomial functions that computes a polynomial representation to the level sets of interest. This second formulation can be solved by using semidefinite programming.
\subsubsection{A direct numerical approach}
\label{sssec:direct}
Given a value of $x$, to compute the probability of success, discretize the space of uncertainty, compute the spread of the $w_d$ that solves $(D_d)$ with respect to $\theta$ as the following
 \begin{align}
   \beta :=\sum_{i=1}^N \min(1,w(x,\theta_i))^k f_\theta(\theta_i)
   \label{eq:beta}
 \end{align}
 where $k\ge 1$, $\{\theta_i,\forall i\in \{1,\ldots,N\}\}$ is the set of discrete values of $\theta$, and $f_\theta(\theta)$ is the density (converted appropriately to a probability mass function) of $\mu_\theta$ with respect to $\lambda_\theta$.
\subsubsection{A more generic method}
Consider the following optimization problem for a given value of $k$
\begin{flalign}
  &&\sup_{q,r} & \phantom{3}\ip{\lambda\otimes \mu_\theta,q} +\ip{\mu_\theta,r} && (PP)\\
  &&\text{st. }&\phantom{3} 0\le q(x,\theta)\le 1                               && \forall (x,\theta)\in X\times \Theta\\
  &&& \phantom{3} q(x,\theta) \le w(x,\theta)^k                                 && \forall (x,\theta)\in X\times \Theta\\
  &&& \phantom{3}0\le r(x)\le \ip{\mu_\theta,q}                                 && \forall x\in X\\
  &&& \phantom{3}r\in \R[x]\\
  &&& \phantom{3}q\in \R[x,\theta]
\end{flalign}
The function $r$ is the function that traces the probability that every $x$ can reach $X_T$.

\section{Examples}
\label{sec:examples}
To solve the following examples, we have adopted the direct method presented in Sec.~\ref{sssec:direct}. All examples used an end time of $T = 1$.
\subsection{1D Constant Dynamics}
\begin{figure}[!t]
  {\includegraphics[trim = 0.00in 0.00in 0.00in 0.00in, clip=true, width=\columnwidth]{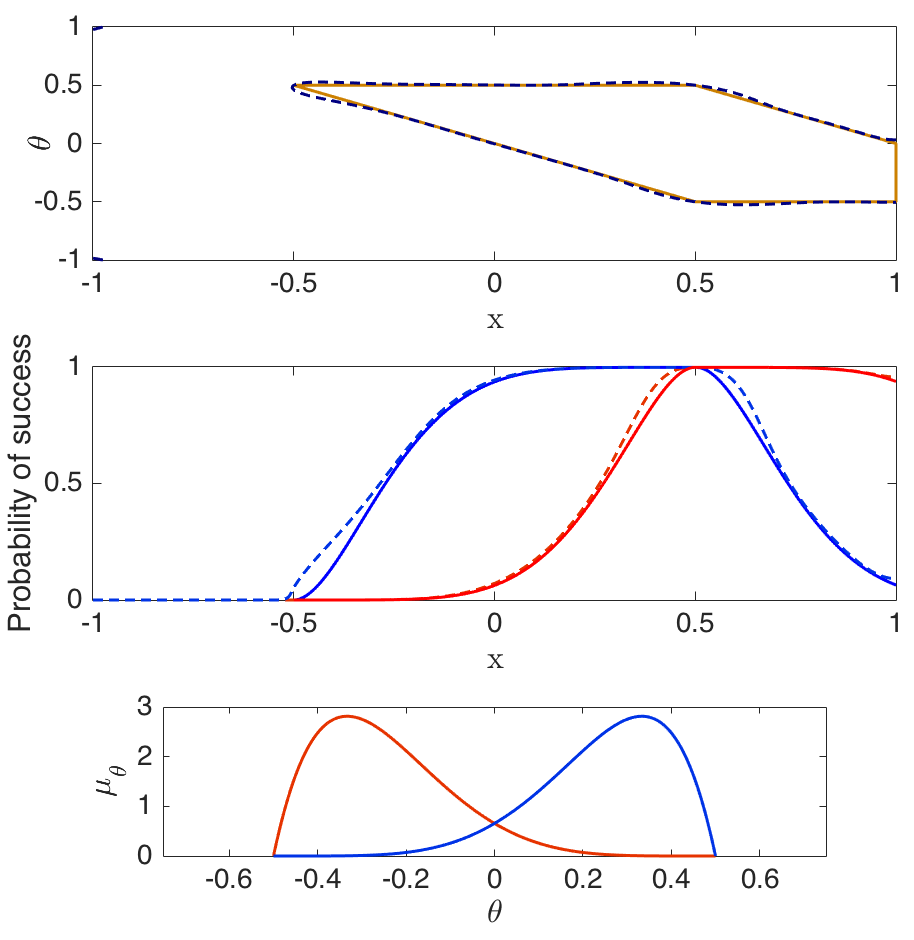}}
  \caption{1D constant dynamics example. Top subplot: BRS on the $x-\theta$ plane. The analytical solution is solid, and the outer approximation is dashed. Middle and bottom subplots: The probability of success across the entire $X$ domain, given two different $\mu_\theta$ distributions, shown with matching colors, the analytical solution in solid lines, and the estimated solution in dashed lines.}
  \label{fig:ex:1D}
\end{figure}
To illustrate the effect of uncertainty directly, consider a 1-dimensional system with constant, but uncertain, dynamics. This can be solved analytically by integrating with respect to $t$:
\begin{align}
  \dot x &= \theta \\
  \Rightarrow x &= \theta t + x_0 \label{1dsol}
\end{align}
where $\theta \in \Theta := [-0.5, 0.5]$ and $x \in X := [-1, 1]$. The target set is $X_T := [0,1]$. Using a pair of right- and left-heavy distributions, $f_1(\theta)$ and $f_2(\theta)$, define the uncertain parameter distribution as:
\begin{align}
  f_1(\theta) = -C(\theta - 0.5)^5(\theta + 0.5) \\
  f_2(\theta) = -C(\theta - 0.5)(\theta + 0.5)^5
\end{align}
where $C$ is chosen to normalize the mass of the distributions on $\Theta$.
\par
From Eqn.~(\ref{1dsol}), the slice of $\spt(\mu_0)$ at any $\theta$ is $x \in [-\theta, 1-\theta]$; the top subplot of Fig.~\ref{fig:ex:1D} shows the $\spt(\mu_0)$ thus computed, in solid lines. The outer approximation of $\spt(\mu_0)$ is plotted in dashed lines. By definition, any vertical slice within $\spt(\mu_0)$ at some $x$ is thus $\Gamma(x)$.
\par
The probability of success is computed as the integral of $w$ with respect to $\mu_\theta$ and is plotted in the second subplot of Fig.~\ref{fig:ex:1D}. The {\em true} probability was computed by discretizing each $f_i, i\in \{1,2\}$ distribution with 600 points and applying Eqn.~(\ref{eq:beta}). The estimated probability was computed with the same discretized distributions, but with $w$ given by the approximate $\spt(\mu_0)$ at each $x$. It is clear from both the top and middle subplots that this provides an outer approximation.
\par
The bottom subplot shows the $f_1$ (red) and $f_2$ (blue) distributions of $\theta$, for reference.

\subsection{Van der Pol Oscillator}
Recall the Van der Pol Oscillator introduced in Sec. II.B:
\begin{align}
  \dot x_1 =&\,    -2x_2\\
  \dot x_2=&\,    0.8x_1+(9+5\theta)x_2(x_1^2-0.21)
\end{align}
where $\theta \in [-0.5,0.5]$, distributed uniformly. $X=[-1,1]$, $X_T=\|x\|\le 0.5$. The uncertain BRS for a degree 12 relaxation is shown in Figure~\ref{fig:ex:vdp}. The estimated $\alpha$ = 1 level set very closely matches the $\alpha$ = 1 level set found through discretization. Effects of the uncertain parameter are most apparent on the top right and bottom left lobes of the uncertain BRS, where the level sets found through the proposed method are separated from those calculated via discretization by a thin strip.
\begin{figure}[!t]
  \includegraphics[trim =1.6in 3.3in 1.6in 3.5in,width=\columnwidth,clip=true]{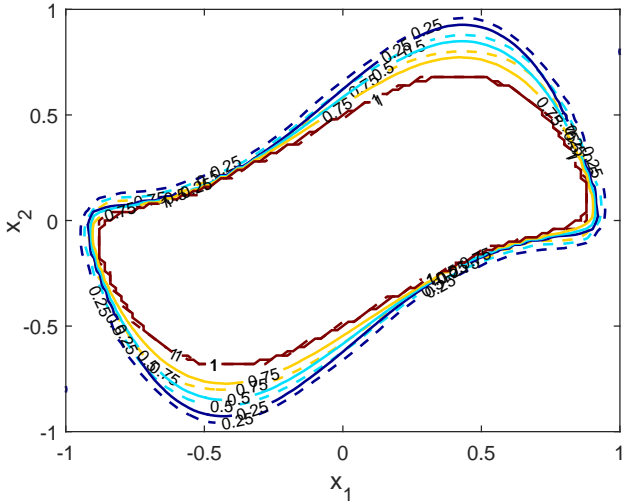}
  \caption{Contours of the level sets corresponding to different probabilities of success; dashed lines are estimates and solid lines are from state-space discretization.}
  \label{fig:ex:vdp}
\end{figure}
\subsection{Ground Vehicle Model}
The Dubins' car \cite{Dubins1957} is commonly used to model ground vehicle behavior, and describes the trajectory of a car's center of mass as a function of its velocity and steering angle. Consider an autonomous vehicle moving in a straight line (horizontally) with a constant velocity, $v = 0.5$ m/s. Suppose that the yaw-rate of the vehicle, $\dot \psi$, is an uncertain parameter, and is denoted by $\theta$. If the uncertain parameter is distributed according to $f_\theta$, the system's dynamics can be described by
\begin{align}
  \dot x = &\,    v\cos(\psi)\\
  \dot y = &\,    v\sin(\psi)\\
  \dot \psi =   &\, \theta
\end{align}
where $x$ and $y$ are the $x$-position, the $y$-position respectively. The dynamics of the vehicle can be represented using polynomials by utilizing the following state transformation \cite{devon2007}:
\begin{align}
  z_1 = &\, \psi,\\
  z_2 = &\, x\cos( \psi ) + y\sin( \psi),\\
  z_3 = &\, x\sin( \psi ) - y\cos ( \psi).
\end{align}
The dynamics of the transformed system are:
\begin{align}
  \dot z_1 = &\,    \theta,\\
  \dot z_2 = &\,     v - z_3\theta,\\
  \dot z_3 = &\,    z_2\theta.
\end{align}
With the above description, the vehicle will travel along trajectories of fixed curvature in the X-Y plane, but it is uncertain which trajectory the car will actually follow. Define a target zone $X_T$ as a ball of radius 0.25 about the origin: $X_T = \|[x;y]\|\leq 0.25$, we solve for the set of initial configurations that can reach the target zone with different probabilities, given that $f_\theta(\theta)$ is given by:
\begin{align}
  f_\theta(\theta) = -C(\theta - 3\pi/4)^3(\theta + 3\pi/4)^3
\end{align}
where $C$ is chosen to normalize the mass of $\mu_\theta$, and $\theta \in [-3\pi/4, 3\pi/4]$.
\par
A degree 14 relaxation was used to determine the uncertainty BRS according to the method proposed in Sec.~\ref{sec:prob}. The resultant $\alpha$-confidence sets were computed as described in Sec.~\ref{sec:implementation} and are presented in Fig.~\ref{fig:ex:dubins}. In order to compute the different confidence level sets, the X-Y plane was discretized into a 201x201 grid, $f_\theta(\theta)$ was discretized into a 501 element vector, and Eqn.~(\ref{eq:beta}) was employed with $k=8$. The true uncertain BRS was determined using Monte Carlo simulations with the same grid, and each node was simulated with 10,000 $\theta$s chosen according to $f_\theta$. The probability of success of any particular node is computed as the proportion of the number of values of $\theta$ for which the resultant trajectory reaches the target zone. From Fig.~\ref{fig:ex:dubins}, it is noted that the estimated $\alpha$-confidence BRS is an outer approximation of the {\em true} $\alpha$-confidence BRS.
\begin{figure}[!t]
\centering
  {\includegraphics[trim =1.5in 2.85in 1.5in 2.9in,width=\columnwidth,clip=true]{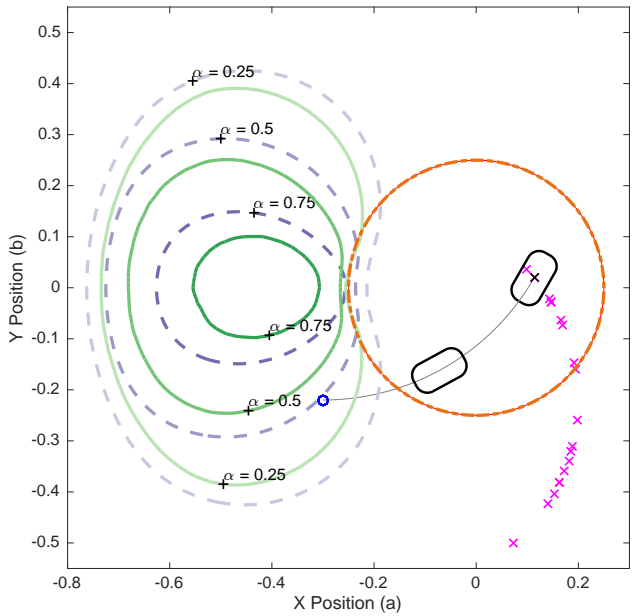}}
  \caption{Contours of the level sets corresponding to different probabilities of success; purple dashed lines are estimates using the proposed method and green solid lines are from state-space discretization. The target set is shown in orange. Xs show endpoints of trajectories emanating from the point [-0.3; -0.22]. One trajectory is plotted, with a small car traveling along it.}
  \label{fig:ex:dubins}
\end{figure}
\par
Figure~\ref{fig:ex:histogram} charts the {\em mean} probability that points on the estimated $\alpha$-confidence BRS reach the target zone. Points were taken from each level set, and each point was forward simulated with 10,000 $\theta$s randomly generated according to $f_\theta$. Atop the histogram, error-bars are overlayed that indicate the 2$\sigma$ band of the probabilities of points on the level set reaching the target zone. Observe that the average probability of success lies below the dashed line with slope one passing through the origin. This indicates that the estimated level sets are indeed outer approximations.
\begin{figure}[!t]
\centering
  {\includegraphics[trim =0.75in 5.5in 1in 2.5in,width=\columnwidth,clip=true]{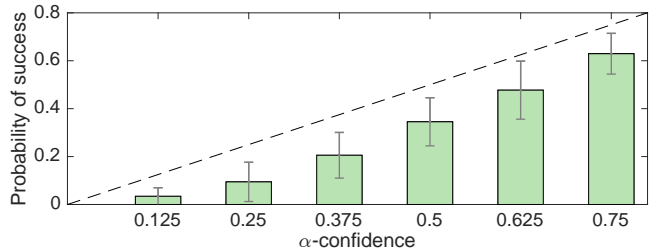}}
  \caption{Histogram showing mean probability of points on the $\alpha$-confidence BRS of reaching the target set. Error bars signify mean $\pm$ standard deviation of level set probability. Dashed line is the ideal probability at each level set.}
  \label{fig:ex:histogram}
\end{figure}

\section{Conclusion}
\label{sec:conclusion}
In this paper, a convex optimization technique to approximate the $\alpha$-confidence backwards reachable set of a parametrically uncertain system is presented. Using the notion of occupation measures, we propose a two step methodology to construct a sequence of convergent approximations of the set of interest -- the first step optimizes over the space of the ring of polynomials with a specified degree and is solved as a sums-of-square program; the second step builds on the result of the first step and constructs an outer approximation of the $\alpha$-confidence reachable set. The proposed method is validated numerically on three examples of varying complexities.

\bibliographystyle{ieeetr}

\end{document}